\documentclass[11pt]{amsart}
\usepackage{xcolor}

\definecolor{grn}{rgb}{0,0.6,0}
\definecolor{mrn}{rgb}{0.3,0,0}
\definecolor{blue}{rgb}{0,0,0.7}
\definecolor{Mygray}{rgb}{0.75,0.75,0.75}
\definecolor{auburn}{rgb}{0.43, 0.21, 0.1}
\definecolor{britishracinggreen}{rgb}{0.0, 0.26, 0.15}
\definecolor{taupe}{rgb}{0.28, 0.24, 0.2}
\usepackage{amsmath,amssymb}
\newtheorem{theorem}{Theorem}[section]
\newtheorem{propn}[theorem]{Proposition}
\newtheorem{cor}[theorem]{Corollary}
\newtheorem{lemma}[theorem]{Lemma}

\newtheorem{rmk}[theorem]{Remark}

\newcommand{\Z}{\mathbb{Z}}
\newcommand{\Q}{\mathbb{Q}}

\newcommand{\Mod}[1]{\ (\mathrm{mod}\ #1)}
\newcommand{\p}{\mathfrak{p}}
\newcommand{\q}{\mathfrak{q}}

\usepackage{latexsym,amssymb}
\usepackage{amssymb}
\usepackage{graphicx}
\usepackage{ textcomp }
\usepackage[left=1.8cm,right=1.8cm,top=1.8cm,bottom=1.8cm]{geometry}
\usepackage{stmaryrd}
\usepackage{tikz}
\usepackage{tikzpagenodes} 
\usepackage{lipsum}
\usepackage{multicol}

\begin{document}
\baselineskip=14.5pt
\title[ Iwasawa module with a bounded  quotient]{ Unramified Iwasawa module of $\Z_2$-extension of certain quadratic fields with a bounded quotient}

\author{H Laxmi and Anupam Saikia}
\address[H Laxmi and Anupam Saikia]{Department of Mathematics, Indian Institute of Technology Guwahati, Guwahati - 781039, Assam, India}

\email[H Laxmi]{hlaxmi@iitg.ac.in}

\email[Anupam Saikia]{a.saikia@iitg.ac.in}
\renewcommand{\thefootnote}{}

\footnote{2020 \emph{Mathematics Subject Classification}: Primary 11R29, Secondary 11R11, 11R23.}

\footnote{\emph{Key words and phrases}: 2-class groups, $\Z_2$-extension of quadratic fields, Iwasawa invariants, structure of Iwasawa module, quotient of Iwasawa module, Greenberg's conjecture.}

\footnote{\emph{We confirm that all the data are included in the article.}}

\renewcommand{\thefootnote}{\arabic{footnote}}
\setcounter{footnote}{0}

\begin{abstract}
We consider an infinite family of real quadratic fields $k$ where the discriminant has three distinct odd prime factors, and the prime 2 splits. We show that the unramified Iwasawa module $X(k_{\infty})$ associated with the $\Z_2$-extension of $k$ has a bounded quotient. Thus, we also verify Greenberg's conjecture on the vanishing of Iwasawa invariants for such fields and obtain a finer structure for  $X(k_{\infty})$.
\end{abstract}

\maketitle
\section{Introduction}
The growth of class groups in an extension of number fields has always been an intriguing subject in algebraic number theory. This theme was explored by K. Iwasawa (cf. \cite{iwasawa1}, \cite{iwasawa}) in an infinite tower of number fields. For a prime number $\ell$, and a number field $K$, he studied the $\ell$-class groups (the $\ell$-Sylow subgroup of the class group) of intermediate number fields present in a $\Z_{\ell}$-extension $K_{\infty}$ of $K$. Let $\Gamma = $ Gal($K_{\infty}/K$), which is topologically isomorphic to $\Z_{\ell}$ as a group. Then for each $n \geq 0$, there exists a unique $K_n \subseteq K_{\infty}$ such that Gal($K_n/K$) $ \cong \Z/{\ell}^n\Z$. Let $A(K_n)$ denote the $\ell$-class group of $K_n$. With respect to the norm maps between class groups in an extension, $\{A(K_n) : n\geq 0\}$ forms an inverse system. We denote the associated inverse limit by $X(K_{\infty})$. If $\Gamma_n = $ Gal($K_n/K$), then $A(K_n)$ is a module over $\Z_{\ell}[\Gamma_n]$, where the action of $\Gamma_n$ on $A(K_n)$ is naturally extended to $\Z_{\ell}[\Gamma_n]$. The {\it Iwasawa algebra} for this $\Z_{\ell}$-extension is defined as $\Lambda = \Z_{\ell}\llbracket \Gamma \rrbracket \cong \lim\limits_{\substack{ \longleftarrow \\ n}} \Z_{\ell}[\Gamma_n]$. It can be seen that (cf. \cite[Proposition 3.2.11]{sharifi} ) $X(K_{\infty})$ is a finitely generated, torsion $\Lambda$-module. By closely studying $X(K_{\infty})$, Iwasawa (cf. \cite[Theorem 11]{iwasawa1}, \cite{iwasawa}) proved that there exist constants $\mu(K_{\infty}/K), \lambda(K_{\infty}/K),$ and $\nu(K_{\infty}/K)$ (known as the {\it Iwasawa invariants}) such that $$e_n = \ell^n\cdot \mu(K_{\infty}/K) + n\cdot\lambda(K_{\infty}/K) + \nu(K_{\infty}/K)$$ for sufficiently large $n$, where $e_n$ denotes the highest power of $\ell$ dividing the order of $A(K_n)$. This equality is known as {\it Iwasawa's class number formula}. When $K$ is totally real, it has a unique $\Z_{\ell}$-extension which is the cyclotomic $\Z_{\ell}$-extension. In that case, we denote the associated Iwasawa invariants by $\mu, \lambda$, and $\nu$. In \cite{greenberg}, Greenberg conjectured that $\mu$ and $\lambda$ must vanish for any totally real field $K$. Ferrero and Washington (cf. \cite{ferrero-washington}) established that $\mu = 0$ for the cyclotomic $\Z_{\ell}$-extension of $K$, when $K$ is an abelian extension of $\Q$. Greenberg's conjecture is still open, with only partial progress made by taking specific prime numbers and families of number fields.

\smallskip

Iwasawa's class number formula allows one to investigate $X(K_{\infty})$ through $A(K_n)$, and vice-versa. Let $L(K_n)$ be the the $\ell$-Hilbert class field of $K_n$, i.e., the maximal unramified abelian $\ell$-extension of $K_n$. Then by Artin map, $A(K_n) \cong$ Gal($L(K_n)/K_n$). If $K_{\infty}^{ur}$ denotes the maximal unramified $\ell$-extension of $K_{\infty}$, then $X(K_{\infty})$ is the largest abelian quotient of Gal($K_{\infty}^{ur}/K_{\infty}$). A noteworthy aspect of a $\Z_{\ell}$-extension is that only the prime(s) above $\ell$ in $K$ can ramify in the extension $K_{\infty}/K$. Moreover, there exists $n_0$ such that any prime ideal above $\ell$ is totally ramified in $K_{\infty}/K_{n_0}$ (\cite[Lemma 13.3]{washington_book}). Let $D(K_n)$ be the subgroup of $A(K_n)$ generated by the ideal classes of prime ideals above $\ell$ in $K_n$. The quotient $A(K_n)/D(K_n)$, denoted by $A^{\prime}(K_n)$ corresponds to the maximal sub-extension $L^{\prime}(K_n)/K_n$ of $L(K_n)/K_n$, where all the primes above $\ell$ split completely. For any $m \geq n \geq 0$, the norm map is well-defined from $D(K_m)$ to $D(K_n)$, and hence, from $A^{\prime}(K_m)$ to $A^{\prime}(K_n)$. Thus $\{D(K_n) : n \geq 0\}$ and $\{A^{\prime}(K_n): n \geq 0\}$ form inverse systems with respect to norm maps. We denote the corresponding inverse limits as $D(K_{\infty})$ and $X^{\prime}(K_{\infty})$. Since $D(K_n)$ is finite for each $n$, it satisfies the Mittag-Lefler condition. Therefore, we can pass the exact sequence 
$$1 \longrightarrow D(K_n) \longrightarrow A(K_n) \longrightarrow A^{\prime}(K_n) \longrightarrow 1$$
to inverse limits and obtain $$1 \longrightarrow D(K_{\infty}) \longrightarrow X(K_{\infty}) \longrightarrow X^{\prime}(K_{\infty}) \longrightarrow 1.$$
In particular, we have $X^{\prime}(K_{\infty})$ as a quotient of $X(K_{\infty})$.

For $\ell = 2$, and any $n \geq 0$, let $\Q_n$ denote the the $n$-th layer of the $\Z_2$-extension of $\Q$. It is the maximal real subfield of $\Q(\zeta_{2^{n+2}})$, where $\zeta_{2^{n+2}}$ is a primitive $2^{n+2}$-the root of unity. For any quadratic field $K = \Q(\sqrt{d})$, the compositum $K\Q_n$ is the $n$-th layer in the $\Z_2$-extension of $K$. Then we have, $\Q_1 = \Q(\sqrt{2})$, $\Q_2 = \Q(\sqrt{ 2 + \sqrt{2}})$, $K_1 = \Q(\sqrt{d}, \sqrt{2})$, $K_2 = \Q(\sqrt{d},\sqrt{ 2 + \sqrt{2}})$, and so on. Because of the comprehensible structure of the number fields present in the extension, and a rich background of genus theory, the case of $\ell =2$ is one of the most explored cases for Greenberg's conjecture. One of the earliest works on verifying that $\lambda = 0$ for $\ell = 2$ was carried out by Ozaki and Taya (cf. \cite{ozaki-taya}). The fields were of the form $\Q(\sqrt{d})$, where $d$ has at most two distinct prime factors with some constraints. Thereafter, a number of mathematicians contributed towards Greenberg's conjecture for $\ell = 2$ by considering $\Q(\sqrt{d})$, where $d$ has at most two or three prime factors satisfying some additional conditions (cf. \cite{fukuda-komatsu}, \cite{kumakawa}, \cite{LS}, \cite{mouhib}, \cite{mouhib-mova}, \cite{mizusawa1}, \cite{Mizusawa ProcAMS}, \cite{Mizusawa CMB}, \cite{nishino}). In most of the fields mentioned there, it can be seen that there is only one prime above $2$. We shall make a slightly different attempt by focusing on an infinite family of real quadratic fields where $2$ splits. Throughout this article, we fix $k = \Q(\sqrt{pqr})$, where $p,q$, and $r$ are distinct odd primes satisfying the following conditions:
\begin{equation}\label{Cond 1}
p \equiv 9\Mod {16},\ q \equiv 3\Mod 8,\ r \equiv 3\Mod 8,\ \left(\dfrac{qr}{p}\right) = -1,\ \left(\dfrac{2}{p}\right)_{4} = -1 \Mod p.
\end{equation}

Here, for $p \equiv 1 \Mod {8}$, $\left(\dfrac{2}{p}\right)_{4}$ denotes the quartic power residue symbol with respect to $p$. The symbol is given by  $\left(\dfrac{2}{p}\right)_{4} \equiv 2^{\frac{p-1}{4}} \equiv \pm 1 \Mod p$. In this article, we shall simultaneously examine both $A(k_n)$ and $A^{\prime}(k_n)$, and obtain the precise structure of $X(K_{\infty})$. We thus prove the following results:

\begin{theorem}\label{X' is bdd}
Let $k = \Q(\sqrt{pqr})$ be a number field, where $p,q$, and $r$ are distinct odd primes satisfying $p \equiv 9\Mod {16},\ q \equiv 3\Mod 8,\ r \equiv 3\Mod 8,\ \left(\dfrac{qr}{p}\right) = -1,\ \left(\dfrac{2}{p}\right)_{4} = -1 \Mod p$. Then, the quotient Iwasawa module $X^{\prime}(k_{\infty})$ is isomorphic to $\Z/2\Z$.
\end{theorem}

\begin{cor}\label{lambda = 0}
The order of $A(k_n)$ is bounded as $n$ tends to infinity. Thus, the Iwasawa invariant $\lambda$ corresponding to the $\Z_2$-extension of $k$ vanishes.
\end{cor}

\begin{cor}\label{4-rank of A(kn)}
For each $n$, the group $A(k_n)$ is isomorphic to $\Z/2\Z \oplus \Z/2^{a_n}\Z$, where $0 \leq a_n \leq n$, and there exists a stage $n_0 \geq 1$ such that $1 \leq a_n = a_{n_0}$ for all $n \geq n_0$. Thus, the $4$-rank of $A(k_n)$ is at most $1$ for all $n \geq 0$. Moreover,  $X(k_{\infty})$ is isomorphic to $\Z/2\Z \oplus \Z/2^{a_{n_0}}\Z$.
\end{cor}

\begin{rmk}
Let $F = \Q(\sqrt{2pqr})$, where the primes $p,q$, and $r$ satisfy Condition (\ref{Cond 1}). The $\Z_2$-extension of $F$ is same as that of $k$, barring the base field. Therefore, the Iwasawa $\lambda$-invariant corresponding to $X(F_{\infty})$ is also equal to $0$, with same the $2$-class group as that of $k_n$ at each level $n \geq 1$.
\end{rmk}
 
\section{Preliminaries}

As we are interested in studying $2$-class groups in the $\Z_2$-extension of $k$, we shall be dealing with their rank and order at multiple levels. For a prime $\ell$, and the $\ell$-class group $A(K_n)$, the $\ell$-rank of $A(K_n)$ is the dimension of $A(K_n)/\ell A(K_n)$ as a vector space over $\Z/\ell \Z$. Thus, it is the number of copies of $\Z/\ell \Z$ present in the cyclic decomposition of $A(K_n)$. We shall use ${\rm{rank}}_{\ell}A(K_{n})$ to denote the $\ell$-rank of $A(K_n)$. We now state a result by Fukuda that gives information about the stability of rank and order of $\ell$-class groups in $\Z_{\ell}$-extension for any prime $\ell$.

\begin{theorem}\cite[Theorem 1]{fukuda}\label{fukuda's result}
Let $\ell$ be a prime number. Let $K$ be a number field and let $K_{\infty}/K$ be a $\Z_{\ell}$-extension of $K$.  Let $A(K_n)$ denote the $\ell$-class subgroup of the $n$-th layer $K_n$ in the extension $K_\infty/K$. Let $n_0 \geq 0$ be an integer such that any prime of $K_{\infty}$ which is ramified in $K_{\infty}/K$ is totally ramified in $K_{\infty}/K_{n_0}$. Then the following hold. 
\begin{enumerate}
\item If there exists an integer $n \geq n_0$ such that $\#A(K_{n+1}) = \#A(K_n)$, then $\#A(K_m) = \#A(K_n)$ for all $m \geq n$. In particular, both the Iwasawa invariants $\mu(K_\infty/K)$ and $\lambda(K_\infty/K)$ vanish. 
  
  \smallskip
  
\item If there exists an integer $n \geq n_0$ such that ${\rm{rank}}_{\ell}A(K_{n+1}) = {\rm{rank}}_{\ell}A(K_{n})$, then ${\rm{rank}}_{\ell}A(K_{m}) = {\rm{rank}}_{\ell}A(K_{n})$ for all $m \geq n$. In particular, the Iwasawa invariant $\mu(K_\infty/K)$ vanishes.
\end{enumerate}
\end{theorem}

Following similar lines of proof, Mizusawa proved an analogous result for the quotient groups $A^{\prime}(K_n)$. We state it as follows:

\begin{theorem}\cite[Proposition 3]{Mizusawa ProcAMS}\label{A' stability}
Let $K$ be a number field, and suppose that the cyclotomic $\Z_{\ell}$-extension $K_{\infty}/K$ is totally ramified at any prime lying over $\ell$. Then, the following hold.
\begin{enumerate}
\item If $\#A^{\prime}(K_{1}) = \#A^{\prime}(K)$, then $\#A^{\prime}(K_n) = \#A^{\prime}(K)$ for all $n \geq 0$. In particular, $\#A^{\prime}(K_n)$ is isomorphic to $\#A^{\prime}(K)$ for all $n \geq 0$. 
  
\smallskip
  
\item If ${\rm{rank}}_{\ell}A^{\prime}(K_{1}) = {\rm{rank}}_{\ell}A^{\prime}(K)$, then ${\rm{rank}}_{\ell}A^{\prime}(K_n) = {\rm{rank}}_{\ell}A^{\prime}(K)$ for all $n \geq 0$. 
\end{enumerate}
\end{theorem}

For a cyclic extension $L/K$, the genus formula is a useful identity that correlates the order of a subgroup of $A(L)$ with that of $A(K)$. The other quantities that are involved are the number of prime ideals ramified in $L/K$, the degree of extension of $L/K$ and certain multiplicative subgroup of units in the ring of integers of $K$ and $L$. We now state the formula in case of a quadratic extension.

\smallskip

\begin{theorem}(\cite{chevalley}, \cite[Section 2.2]{Mizusawa ProcAMS})
Let $L/K$ be a quadratic extension with Galois group $G$. Let $t$ be the number of prime ideals ramified in $L/K$. Let $A(L)^{G}$ be the subgroup of $A(L)$ consisting of the ideal classes that are fixed by the action of $G$ on $A(L)$. Let $N_{L/K}$ denote the norm map from $L$ to $K$, and $E(L)$ and $E(K)$ be the unit groups of the ring of integers of $L$ and $K$, respectively. Then we have,
\begin{align*}
\#A(L)^{G} &= \#A(K) \times \dfrac{2^{t-1}}{\left[ E(K):E(K)\cap N_{L/K}L^{\times} \right]}.
\end{align*}
\end{theorem}

If $A(K)$ is trivial, then the norm map from $A(L)$ to $A(K)$ is trivial and $G$ acts as $[\mathfrak{a}] \mapsto [\mathfrak{a}]^{-1}$ on $A(L)$. In that case, $A(L)^{G}$ is isomorphic to $A(L)/2A(L)$, and we obtain that $\#A(L)^{G} = 2^{\  {\rm{rank_2}} A(L) }$. For a quadratic extension $L/\Q$, the field corresponding to $A(L)^G$ is known as the genus field of $L$, denoted by $L_G$. It is the maximal abelian extension of $\Q$, contained in the 2-Hilbert class field of $L$. Suppose $D_L = 2^ep_1\cdots p_r$ ($e = 0,2$ or $3$) is the discriminant of $L$, where $p_i$'s are distinct odd primes. Let $L_G^{+}$ be the field $\Q(\sqrt{D_L}, \sqrt{p_1^*},\ldots,\sqrt{p_r^*})$, where $p_i^* = p_i$ or $-p_i$ in case of $p_i \equiv 1 \Mod{4}$ and $p_i \equiv 3 \Mod{4}$ respectively. If $L$ is real, then $L_G$ is the maximal real subfield of $L_G^{+}$, otherwise, it is equal to $L_G^{+}$ (for details, cf. \cite[Chapter 6]{janusz}).

\smallskip

From the genus formula, it can be seen that the unit group and norm of elements in an extension $L/K$ play a key role in the structure of $A(L)$. Therefore, it is vital to study these ingredients, and we do so by using the {\it Hilbert symbol}. In a finite abelian extension $L/K$, for $\alpha \in K$ and a prime ideal $\mathfrak{P}$, the {\it norm residue symbol} of $\alpha$ with respect to $\mathfrak{P}$ is an element in Gal($L/K$). Suppose $\mathfrak{P}$ is an unramified prime ideal in $L/K$, and $\langle \alpha \rangle = \mathfrak{P}^a$ for some integer $a \geq 0$ in the completion $K_{\mathfrak{P}}$. Then the norm residue symbol of $\alpha$ with respect to $\mathfrak{P}$ is defined as
$$\left(\dfrac{\alpha, L/K}{\mathfrak{P}}\right) = \left(\dfrac{L/K}{\mathfrak{P}}\right)^a,$$
where $\left(\dfrac{L/K}{\mathfrak{P}}\right)$ is the Frobenius element of $\mathfrak{P}$ corresponding to $L/K$. This symbol is also defined for ramified prime ideals via ideles in class field theory (for further details, see \cite{nancy_childress}, \cite{RoyCFT}, \cite{janusz}). Suppose $K$ contains the $m$-th roots of unity, and $L = K(\beta^{1/m})$ for some $\beta \in K^{\times}\setminus (K^{\times})^m$. Then for any $\alpha \in K$ and a prime ideal $\mathfrak{P}$ in $L/K$, the action of the norm residue symbol on $\beta^{1/m}$ is given by $\left(\dfrac{\alpha, L/K}{\mathfrak{P}}\right)\cdot \beta^{1/m} = \zeta_m\cdot\beta^{1/m}$, where $\zeta_m$ is a $m$-th root of unity. This root of unity $\zeta_m$ is denoted by $\left( \dfrac{\alpha, \beta}{\mathfrak{P}} \right)$, and is known as the Hilbert symbol. One of the most significant characteristics of the Hilbert symbol is that an element $\alpha \in K$ is a norm in the extension $K(\beta^{1/m})/K$ if and only if $\left( \dfrac{\alpha, \beta}{\mathfrak{P}} \right)=1$ for all the prime ideals $\mathfrak{P}$ of $K$. The Hilbert symbol possesses a number of interesting properties. The ones that we require in this work are the following (cf. \cite[Lemma 10.6]{RoyCFT}):
\begin{multicols}{2}
\begin{enumerate}
\item $\left( \dfrac{\alpha, \beta\gamma}{\mathfrak{P}} \right)= \left( \dfrac{\alpha, \beta}{\mathfrak{P}} \right)\left( \dfrac{\alpha, \gamma}{\mathfrak{P}} \right)$
\item $\left( \dfrac{\alpha, \beta}{\mathfrak{P}} \right) = \left( \dfrac{\beta, \alpha}{\mathfrak{P}} \right)^{-1}$.
\end{enumerate}
\end{multicols}
\noindent As we shall be only dealing with real number fields, the only roots of unity belonging to $\Q_n$ and $k_n$ are $-1$ and $1$ for any $n$. For that reason, the Hilbert symbols that arise here will have value $-1$ or $1$.

\smallskip

For a totally real number field $K$, suppose that the $\Z_{\ell}$-extension $K_{\infty}/K$ is topologically generated by $\gamma$. For each $n$, $\gamma$ induces a map on $K_n$ via restriction, which we denote by $\tau_n$. The extension $K_n/K$ is cyclic and its Galois group is generated by $\tau_n$. Let $B(K_n)$ be the group
$\{ [\mathfrak{a}] \in A(K_n) : [\mathfrak{a}]^{\tau_n}= [\mathfrak{a}]\}$. We observe that if we apply the genus formula on the cyclic extension $K_n/K$, then the formula yields the order of $B(K_n)$. We shall use this group at a later stage. We now state an important result proved by Greenberg regarding the order of $B(K_n)$.

\begin{propn}\cite[Proposition 1]{greenberg}\label{Bn bdd}
Let $K$ be a totally real number field for which Leopoldt's conjecture holds true. Then, the order of $B(K_n)$ remains bounded as $n$ tends to infinity.
\end{propn}

Greenberg has expressed various formulations of Leopoldt's conjecture in \cite[Page 265]{greenberg}. In \cite{Brumer}, Brumer validated Leopoldt's conjecture for abelian extensions of $\Q$. As the field $k$ under our consideration is a real quadratic extension of $\Q$, Leopoldt's conjecture holds true for $k$, and hence, the requirement in Proposition \ref{Bn bdd} is fulfilled by $k$.

\section{Structure of $A(k_n)$}
We consider $k = \Q(\sqrt{pqr})$ such that $p$, $q$ and $r$ satisfy Condition (\ref{Cond 1}). From the congruence conditions on $p,q$, and $r$, the discriminant $D_k$ of $k$ is equal to $pqr$, and the prime above $2$ is split in $k_n/\Q_n$ for all $n \geq 0$. As $D_k$ has a prime factor congruent to 3 modulo 8, $-1$ is not a norm in the extension $k/\Q$. Since $h(\Q) = 1$, the genus formula implies that  ${\rm{rank}}_{2}A(k)=1$. Additionally, the genus field $k_G$ of $k$ is given by $k(\sqrt{p}) = \Q(\sqrt{p},\sqrt{qr})$. We now recall a lemma for a larger family of number fields $K$ containing the family $k$ under our consideration. This will allow us to conclude that the order of $A(k)$ is $2$. 

\begin{lemma}\cite[Lemma 3.1]{LS}\label{A_k = 2}
Let $K = \Q(\sqrt{pqr})$ such that $p \equiv 1 \Mod{4}$ and $q,\; r \equiv 3 \Mod{4}$. Then, $\#A(K) = 2$ if and only if $-1 \in \left\lbrace \left(\dfrac{q}{p}\right), \left(\dfrac{r}{p}\right) \right\rbrace $.
\end{lemma}

The quadratic subfields of a bi-quadratic extension govern the properties of the class group of the bigger field to a large extent. In case of a quadratic field $K = \Q(\sqrt{d})$, the first layer in its $\Z_2$-extension is $K_1 = \Q(\sqrt{2}, \sqrt{d})$. Thus, before moving to the first layer, it is imperative to inspect $\Q(\sqrt{2d})$. Let $F = \Q(\sqrt{2pqr})$, where the primes $p,q$, and $r$ satisfy Condition \ref{Cond 1}. We now prove the following lemma which enables us to determine the structure of $A(F)$.

\begin{lemma}\label{A(F) = (2,2)}
Let $F = \Q(\sqrt{2pqr})$ with $p \equiv 1\Mod{8}$, $q,r\equiv 3\Mod{8}$, and $\left(\dfrac{qr}{p}\right) = -1$. Then, $A(F)$ is isomorphic to $\Z/2\Z \oplus \Z/2\Z$.
\end{lemma}
\begin{proof}
From the genus formula, the 2-rank of $A(F)$ is $2$, and its genus field is $F_G = \Q(\sqrt{2}, \sqrt{p}, \sqrt{qr})$. The three quadratic subfields of $F_G$ containing $F$ are $F(\sqrt{2})$, $F(\sqrt{p})$, and $F(\sqrt{qr})$. Let $\p^{\prime}, \q^{\prime}$ and $\mathfrak{r}^{\prime}$ be the prime ideals above $p, q$ and $r$ in $F$. Without loss of generality, suppose $\left( \dfrac{q}{p} \right) = \left( \dfrac{p}{q} \right) = -1$. The other case can be dealt with in a similar manner. In this case, $q$ is inert in $\Q(\sqrt{p})/\Q$, and thus, $\q^{\prime}$ is inert in $F(\sqrt{p})/F$. Likewise, $\mathfrak{r}^{\prime}$ and $\p^{\prime}$ are inert in $F(\sqrt{qr})/F$, as $r$ and $p$ are inert in $\Q(\sqrt{2p})/\Q$ and $\Q(\sqrt{qr})/\Q$ respectively. As these ideals are not totally split in $F_G/F$, they cannot be totally split in $L(F)/F$, where $L(F)$ is the $2$-Hilbert class field of $F$. Hence, they are not principal by class field theory. Thus, the classes $[\p^{\prime}], [\q^{\prime}]$ and $[\mathfrak{r}^{\prime}]$ are non-trivial, and of order $2$ in $A(F)$. Also, the decomposition fields of the primes $\p^{\prime}, \q^{\prime}$ and $\mathfrak{r}^{\prime}$ in $F_G/F$ are different, so their ideal classes must be mutually distinct.

\smallskip

Let Gal($F/\Q$) = $\langle \phi \rangle$. Then by Artin map, $A(F)/A(F)^{\phi -1}$ is isomorphic to Gal($F_G/F$). Since $\Q$ has class number $1$, $\phi$ acts as $-1$ on $A(F)$, and hence, $A(F)^{\phi -1} = A(F)^2$. As $\p^{\prime}, \q^{\prime}$ and $\mathfrak{r}^{\prime}$ are not split in $F_G$, the image of their ideal classes must be non-trivial in Gal($F_G/F$). This yields that $[\p^{\prime}], [\q^{\prime}]$ and $[\mathfrak{r}^{\prime}]$ do not belong to $A(F)^{\phi -1} = A(F)^2$, and Gal($F_G/F$) $ \cong A(F)/A(F)^{\phi -1} = A(F)/A(F)^2 \cong A(F)[2] = \langle [\p^{\prime}], [\q^{\prime}] \rangle$. As $[\p^{\prime}], [\q^{\prime}] \not\in A(F)^2$, $A(F) = \langle [\p^{\prime}], [\q^{\prime}] \rangle \cdot A(F)^2$. By Nakayama's lemma, $A(F) = \langle [\p^{\prime}], [\q^{\prime}] \rangle$, which is isomorphic to $\Z/2\Z \oplus \Z/2\Z$. 
\end{proof}

We shall now proceed to study $A(k_1)$ by using the genus formula for the extension $k_1/\Q_1$. In order to calculate the rank of $A(k_1)$, we need to determine the index $[E(\Q_1): E(\Q_1)\cap N_{k_1/\Q_1}k_1^{\times}]$. This is the order of an elementary 2-group, as the square of any element of $E(\Q_1)$ is clearly a norm in the extension $k_1/\Q_1$. We note that $E(\Q_1)$ is generated by $-1$ and $1 + \sqrt{2}$, and prove the following results.

\begin{propn}\label{not a norm in Q1} 
Let $K = \Q(\sqrt{d})$ where $d$ has a prime factor congruent to 3 modulo 4. Then, for $K_1 = K(\sqrt{2})$, $1 + \sqrt{2}$ is not a norm in the extension $K_1/\Q_1$.
\end{propn}
\begin{proof}
Suppose there exists an element $\alpha$ in $K_1 = \Q(\sqrt{2},\sqrt{d})$ such that $N_{K_1/\Q_1}(\alpha) = 1+\sqrt{2}$. Then we have $N_{\Q_1/\Q}\left(N_{K_1/\Q_1}(\alpha)\right) = -1$. On the other hand, since $K_1$ is bi-quadratic over $\Q$, $N_{\Q_1/\Q}\left(N_{K_1/\Q_1}(\alpha)\right) = N_{K/\Q}\left(N_{K_1/K}(\alpha)\right)$. This means that there is an element in $K$, whose norm is equal to $-1$ over $\Q$. This is not possible as $d$ has a prime factor congruent to 3 modulo 4. Thus, the proposition follows.
\end{proof}

\begin{lemma}\label{rank of A(k_1)}
In the extension $k_1/\Q_1$, $-1$ is a norm, and ${\rm{rank}_2} A(k_1) = 2$.
\end{lemma}
\begin{proof}
Since $k_1 = \Q_1(\sqrt{pqr})$, $-1$ is a norm in $k_1/\Q_1$ if and only if $\left(\dfrac{-1, pqr}{\mathfrak{P}}\right) = 1$ for every prime ideal $\mathfrak{P}$ of $\Q_1$. Since $-1$ is a unit, it is obvious that $\left(\dfrac{-1, pqr}{\mathfrak{P}}\right) = 1$ for any prime ideal $\mathfrak{P}$ of $\Q_1$ which is unramified in $k_1$. Therefore, we need to examine the symbol with respect to the prime ideals above $p, q$, and $r$ in $\Q_1$. As $p \equiv 1\Mod 8$, $p$ splits completely in $\Q_1$, and $\langle p \rangle = \p_1\p_2$, where $\p_1$ and $\p_2$ are prime ideals dividing $p$ in $\Q_1$. Owing to the multiplicative property of the Hilbert symbol, for $i = 1,2$, we have $$\left( \dfrac{-1, pqr}{\p_i}\right) =  \left( \dfrac{-1, p}{\p_i}\right)\left( \dfrac{-1, q}{\p_i}\right)\left( \dfrac{-1, r}{\p_i}\right).$$ 
The primes $\p_i$ are unramified in the extension $\Q_1(\sqrt{q})/\Q_1$ and $\Q_1(\sqrt{r})/\Q_1$. Therefore, for $i = 1$ and $2$, $$\left( \dfrac{-1, q}{\p_i}\right)=\left( \dfrac{-1, r}{\p_i}\right)=1, \text{ and } \left( \dfrac{-1, pqr}{\p_i}\right)=\left( \dfrac{-1, p}{\p_i}\right).$$
As $\left( \dfrac{-1, p}{\p_i}\right) = \left( \dfrac{p, -1}{\p_i}\right)^{-1}$, we consider the extension $\Q_1(\sqrt{-1})/\Q_1$. Since $p \equiv 1 \Mod{8}$, $p$ splits completely in $\Q_1/\Q$ and $\Q(\sqrt{-2})/\Q$. Thus, the primes $\p_i$ are totally split in $\Q_1(\sqrt{-1})/\Q_1$, and the corresponding Frobenius elements are equal to the identity map. As a result, $$\left( \dfrac{p, -1}{\p_i}\right)^{-1} = \left( \dfrac{-1, p}{\p_i}\right) = \left( \dfrac{-1, pqr}{\p_i}\right) = 1.$$ Since $ q \equiv 3\Mod{8}$, $\langle q \rangle$ is a prime ideal in $\Q_1$. As before, we have $$\left( \dfrac{-1, pqr}{\langle q \rangle}\right) = \left( \dfrac{-1, p}{\langle q \rangle}\right) \left( \dfrac{-1, q}{\langle q \rangle}\right) \left( \dfrac{-1, r}{\langle q \rangle}\right) =  \left( \dfrac{-1, q}{\langle q \rangle}\right)=  \left( \dfrac{q, -1}{\langle q \rangle}\right)^{-1}.$$ Now, $\left( \dfrac{-2}{q} \right) = 1$ leads to $\langle q \rangle$ being split in $\Q(\sqrt{-2})/\Q$, and thus in $\Q_1(\sqrt{-1})/\Q_1$. Again, the Frobenius element corresponding to $\langle q \rangle$ in $\Q_1(\sqrt{-1})/\Q_1$ is the identity map. It follows immediately that $\left( \dfrac{-1, pqr}{\langle q \rangle}\right) = 1$. The same argument holds for $\langle r \rangle$. Consequently, we have $\left( \dfrac{-1, pqr}{\langle r \rangle}\right) = 1$. Therefore, the Hilbert symbol $\left( \dfrac{-1, pqr}{\mathfrak{P}}\right)$ has value $1$ for all the prime ideals $\mathfrak{P}$ of $\Q_1$, and hence, $-1$ is a norm in $k_1/\Q_1$. Given that $-1$ is a norm and $1+\sqrt{2}$ is not a norm (from Proposition \ref{not a norm in Q1}) in $k_1/\Q_1$, we conclude that $-(1+\sqrt{2})$ is also not a norm in $k_1/\Q_1$. The number of primes ramified in $k_1/\Q_1$ is equal to $4$. Employing these facts in the genus formula for $k_1/\Q_1$, we obtain,
$2^{\ {\rm{rank}_2}A(k_1)} = \frac{2^{4-1}}{2} = 2^2$. Hence, ${\rm{rank}_2}A(k_1) = 2$.
\end{proof}

In \cite{kumakawa}, Kumakawa studied the $\Z_2$-extension of a different family of real quadratic fields $K = \Q(\sqrt{pq})$. He derived a relation between the order of $A(K_{n+1})$ and the orders of the 2-class groups of some of the subfields of $K_{n+1}$ containing $\Q_n$. This relation was further used to estimate the order of $A(K_2)$. We now show that when $K=\Q(\sqrt{d})$, with $d \equiv 1\Mod{4}$, a finer relation can be deduced  for the $\Z_2$-extension of $K$. Let $K_n, K_{n}^{\prime}$ and $\Q_{n+1}$ be the three subfields of $K_{n+1}$ containing $\Q_n$. For $n=0$, $K_0^{\prime} = \Q(\sqrt{2d})$, for $n=1$, $K_1^{\prime} = \Q\left(\sqrt{(2+\sqrt{2})d}\right)$, and so on. For any fixed $n \geq 0$, suppose $\rho$ is the generator of \rm{Gal}$(K_{n+1}/K_n)$ and $\sigma$ is the generator of \rm{Gal}$(K_{n+1}/\Q_{n+1})$. Since $K_{n+1}/\Q_n$ is a bi-quadratic extension, \rm{Gal}$(K_{n+1}/K_n^{\prime})$ must be generated by $\rho\sigma$. The proof follows similarly as given in \cite[Lemma 2.1]{kumakawa} with some modification at the end, which we present here for the sake of completeness. 

\begin{propn}\label{Bound on A_{n+1}}
Let $K= \Q(\sqrt{d})$, where $d \equiv 1\Mod{4}$. Then, $\#A(K_{n+1}) \leq \#A(K_{n+1})^{\rho + 1}\cdot\#A(K_n^{\prime})/2.$ In particular, $\#A(K_{n+1}) \leq \#A(K_n)\cdot\#A(K_n^{\prime})/2$.
\end{propn}
\begin{proof}
The map $\sigma$ acts as $-1$ on $K_{n+1}$ as the class number of $\Q_{n+1}$ is odd (cf. Theorem 10.4, \cite{washington_book}). Therefore, the sets  $A(K_{n+1})^{\sigma\rho -1} = \{ [\mathfrak{a}]^{\sigma\rho}\cdot[\mathfrak{a}]^{-1}: [\mathfrak{a}] \in A(K_{n+1})\}$ and $A(K_{n+1})^{\rho+1} = \{ [\mathfrak{a}]^{\rho}\cdot[\mathfrak{a}]: [\mathfrak{a}] \in A(K_{n+1})\}$ are equal. Combining this equality with the exact sequence
$$1 \longrightarrow A(K_{n+1})^{ \langle \sigma\rho \rangle} \longrightarrow A(K_{n+1}) \longrightarrow A(K_{n+1})^{\sigma\rho -1} \longrightarrow 1,$$ we obtain $\#A(K_{n+1}) = \#A(K_{n+1})^{ \langle \sigma\rho \rangle}\cdot\#A(K_{n+1})^{\rho +1}$. As $d \equiv 1 \Mod{4}$, the extension $K_{n+1}/K_n^{\prime}$ is unramified. Applying the genus formula to the quadratic extension $K_{n+1}/K_n^{\prime}$, we obtain $\#A(K_{n+1})^{ \langle \sigma\rho \rangle} \leq \#A(K_{n}^{\prime})/2$. In addition, $1 + \rho$ acts as the norm map from $A(K_{n+1})$ to $A(K_n)$. It follows that $\#A(K_{n+1})^{\rho + 1} \leq \#A(K_n)$. Both the relations mentioned in the statement of the lemma follow immediately from these observations.  
\end{proof}

\begin{cor}\label{A(k1)=(2,2)}
The group $A(k_1)$ is isomorphic to $\Z/2\Z \oplus \Z/2\Z$.
\end{cor}
\begin{proof}
It is clear that $k_0^{\prime}$ is a sub-class of fields $F$ mentioned in Lemma \ref{A(F) = (2,2)}. The group $A(k_1)$ has rank 2 from Lemma \ref{rank of A(k_1)}. From Proposition \ref{Bound on A_{n+1}} and lemmas \ref{A_k = 2}, \ref{A(F) = (2,2)}, $\#A(k_1) \leq \#A(k_0)\cdot\#A(k_0^{\prime})/2 = 2\cdot 4/2 = 4$. Hence, its order must be exactly equal to 4, and the group must be isomorphic to $\Z/2\Z \oplus \Z/2\Z$.
\end{proof}

As the next step, we shall prove the following lemma on the rank of the 2-class groups of the subsequent fields $k_n$ in the $\Z_2$-extension of $k$ for $n \geq 2$.

\begin{lemma}\label{rank A(Kn)=2}
The 2-rank of $A(k_2)$ is equal to 2. Further, ${\rm{rank}_2} A(k_n) = 2$ for all $n\geq 1$.
\end{lemma}
\begin{proof}
As the class numbers of $\Q_1$ and $\Q_2$ are odd, it is clear from the genus formula that the norm map from $\Q_2^{\times}$ to $E(\Q_1)$ is surjective. Thus, there exists an element $\alpha$ in $\Q_2$, whose norm over $\Q_1$ is equal to $1+\sqrt{2}$. Now suppose (if possible), that there exists a $u$ in $k_2^{\times}$ such that $N_{k_2/\Q_2}(u) = \alpha$. Then, taking norm over $\Q_1$, we get $N_{\Q_2/\Q_1}(N_{k_2/\Q_2}(u)) = N_{\Q_2/\Q_1}(\alpha) = 1+\sqrt{2}$. But this leads to  $N_{k_2/\Q_1}(u) = N_{k_1/\Q_1}(N_{k_2/k_1}(u)) = 1 + \sqrt{2}$, which  shows that $k_1$ has an element of norm $1+\sqrt{2}$ over $\Q_1$. This is a contradiction to Proposition \ref{not a norm in Q1}, and hence, no such $u$ exists in $k_2$. Therefore, $E(\Q_2) \neq E(\Q_2)\cap N_{k_2/\Q_2}(k_2^{\times})$, and $[E(\Q_2): E(\Q_2)\cap N_{k_2/\Q_2}(k_2^{\times})] \geq 2$.

Since $p \equiv 9 \Mod {16}$, the primes $\p_1$ and $\p_2$ of $\Q_1$ lying above $\langle p \rangle$ remain inert in $\Q_2/\Q_1$. Therefore, only four primes of $\Q_2$ are ramified in $k_2/\Q_2$. Incorporating this with the aforementioned conclusion in the genus formula for $k_2/\Q_2$, we obtain $ {\rm{rank}_2}A(k_2) \leq 2$. The primes above $2$ are totally ramified in $k_2/k_1$ and hence, the norm map from $A(k_2)$ to $A(k_1)$ is surjective. Consequently, ${\rm{rank}_2}A(k_2) \geq {\rm{rank}_2}A(k_1) = 2$, and thus, ${\rm{rank}_2}A(k_2) =2 = {\rm{rank}_2}A(k_1)$. The primes above $2$ are totally ramified in $k_{\infty}/k$, and hence in $k_{\infty}/k_1$. Therefore, from part $(2)$ of Theorem \ref{fukuda's result}, ${\rm{rank}_2}A(k_n) = 2$ for all $n \geq 1$.
\end{proof}

\section{Structure of $A^{\prime}(k_n)$}
In this section, we shall primarily focus on the groups $A^{\prime}(k_0)$ and $A^{\prime}(k_1)$, with the objective of gathering information on their orders. 

\begin{propn}
Suppose $K = \Q(\sqrt{pqr})$ with $p \equiv 1 \Mod{8},\ q,r \equiv 3\Mod 4$ such that $qr \equiv 1\Mod{8}$, and $-1 \in \left\lbrace \left(\dfrac{q}{p}\right), \left(\dfrac{r}{p}\right) \right\rbrace$. Then, $\#A^{\prime}(K) = 2$.
\end{propn} 
\begin{proof}
From Lemma \ref{A_k = 2}, $\#A(K) = 2$ and the 2-Hilbert class field of $K$ is same as its genus field $K_G = Q(\sqrt{p}, \sqrt{qr})$. Since $p \equiv 1 \Mod{8}$ and $qr \equiv 1\Mod{8}$, the prime $2$ splits completely in the field extensions $\Q(\sqrt{p})/\Q$ and $\Q(\sqrt{qr})/\Q$. Thus, the primes above $2$ in $K$ split completely in $K_G$. Therefore by definition, $A^{\prime}(K)$ is isomorphic to Gal($K_G/K$), which is of order $2$.  
\end{proof}

\begin{rmk}\label{A'(k0) =2}
The family of fields $k$ adhering to Condition (\ref{Cond 1}) forms
a sub-class of fields mentioned in the proposition above. Hence, for such $k$, $\#A^{\prime}(k) = \#A^{\prime}(k_0)  = 2$. Let $\ell_{n1}$ and $\ell_{n2}$ be the the prime ideals above $2$ in $k_n$ for any $n$. For $n=0$, as $\ell_{01}$ and $\ell_{02}$ are totally split in the 2-Hilbert class field of $k_0$, these ideals must be principal, i.e, $D(k_0)$ is the trivial group.
\end{rmk}

\noindent As $p \equiv 9 \Mod {16}$, we have $\langle p \rangle = \p_1\p_2$ in $\Q_1$. These prime ideals are principal as $\Q_1$ is a principal ideal domain. The generators of these ideals can be chosen in such a way that they are totally positive. Let $p_1$ and $p_2$ be such elements in $\Q_1$. From \cite[page no. 242]{nishino}, we obtain that $p_i \equiv \pm 3, \pm (1+2\sqrt{2}) \Mod{4\sqrt{2}}$, for $i = 1,2$. 
Our aim is to calculate the rank and order of $A^{\prime}(k_1)$. As a preparation, we first refer to some results that are employed frequently to characterize extensions of $\Q_1$ that are ramified and unramified at certain primes (cf. \cite{fukuda-komatsu}, \cite{kumakawa}, \cite{Mizusawa ProcAMS}, \cite{nishino}). These are further used to determine unramified extensions of $k_1$. 
\begin{lemma}\cite[Lemma 5]{Mizusawa ProcAMS}\label{cong condn for unram 2}
 Let $\alpha \in \Q_1$ be a non-square, coprime to $2$. Then, the following hold:
 \begin{enumerate}
     \item The ideal $\langle \sqrt{2} \rangle$ is unramified in $\Q_1(\sqrt{\alpha})/\Q_1$ if and only if $\alpha \equiv 1 \text{ or } 3 +2\sqrt{2} \Mod{4}$.
     \item The ideal $\langle \sqrt{2} \rangle$ is totally split in $\Q_1(\sqrt{\alpha})/\Q_1$ if and only if $\alpha \equiv 1 \text{ or } 3 +2\sqrt{2} \Mod{4\sqrt{2}}$.
 \end{enumerate} 
\end{lemma}

\begin{lemma}\cite[Lemma 1]{nishino}\label{quartic condn for unr outside p1}
Let $p \equiv 9\Mod{16}$ be a prime with a factorization $p = p_1p_2$ in $\Q_1$ where $p_1$ and $p_2$ are totally positive, prime elements in $\Q_1$. Then, there exists a quadratic extension of $\Q_1$, unramified outside $p_i$ if and only if $\left( \dfrac{2}{p} \right)_{4} \neq 1\Mod{p}$.   
\end{lemma}

\begin{propn}\label{2 is inert in Q1(p1)}
   Let $p \equiv 9 \Mod{16}$ be a prime number with a factorization $p = p_1p_2$ in $\Q_1$ where $p_i, i=1,2$ are totally positive, prime elements in $\Q_1$. If $\left(\dfrac{2}{p} \right)_{4} = -1$, then the prime ideal $\langle \sqrt{2}\rangle$ is inert in $\Q_1(\sqrt{p_i})/\Q_1$. 
\end{propn}
\begin{proof}
    We furnish a proof for $p_1$ as the same holds for $p_2$. As $\left(\dfrac{2}{p} \right)_{4} = -1$, $\Q_1$ has a quadratic extension unramified outside $p_1$ (from Lemma \ref{quartic condn for unr outside p1}). Let that extension be $\Q_1(\sqrt{\alpha})$ for some non-square $\alpha \in \Z[\sqrt{2}]$. As the extension is unramified outside $p_1$, it must be unramified at $\langle \sqrt{2} \rangle$, and it has to be a real extension. Also, since the class number of $\Q_1$ is equal to $1$, any non-trivial extension of $\Q_1$ has to be ramified at some prime. Putting these together, we obtain that $\Q_1(\sqrt{\alpha}) = \Q_1(\sqrt{p_1})$, along with $p_1 \equiv 1 \text{ or } 3 +2\sqrt{2} \Mod{4}$ (from Lemma \ref{cong condn for unram 2}, part 1). We already have that $p_1 \equiv \pm 3 \text{ or } \pm(1+2\sqrt{2}) \Mod{4\sqrt{2}}$. If $p_1 \equiv 3$ or $(1+2\sqrt{2}) \Mod{4\sqrt{2}}$, then $p_1 \equiv 3$ or $(1+2\sqrt{2}) \Mod{4}$, which is not possible. Therefore, $p_1 \equiv -3$ or $-(1+2\sqrt{2}) \Mod{4\sqrt{2}}$, which is not congruent to any of $1$ or $3 + 2\sqrt{2} \Mod{4\sqrt{2}}$. Thus, $\langle \sqrt{2} \rangle$ is inert in the extension $\Q_1(\sqrt{p_1})/\Q_1$.
\end{proof}

\subsection*{Proof of Theorem \ref{X' is bdd}}
Let $k$ be a real quadratic field satisfying Condition \ref{Cond 1}. Let $p = p_1p_2$ be a factorization of $p$ in $\Q_1$ as given in Proposition \ref{2 is inert in Q1(p1)}. Then, $p_2\equiv -3$ or $-(1+2\sqrt{2}) \Mod{4\sqrt{2}}$. Since $qr \equiv 1 \Mod{8}$, we obtain $qr \equiv 1\Mod{4\sqrt{2}}$, $p_2qr \equiv 1 \text{ or } 3+2\sqrt{2} \Mod{4}$ and $p_2qr \not\equiv 1 \text{ or } 3+2\sqrt{2} \Mod{4\sqrt{2}}$. Therefore, $\langle \sqrt{2} \rangle$ is inert in $\Q_1(\sqrt{p_2qr})/\Q_1$ from Lemma \ref{cong condn for unram 2}, part (2). Consider the fields $k_1(\sqrt{p_1})$ and $k_1(\sqrt{p})$. We now claim that these are unramified over $k_1$.
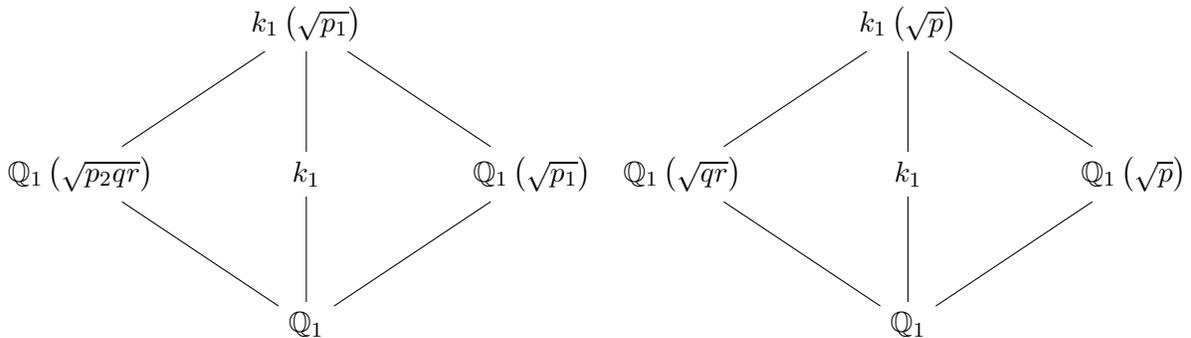
\begin{center}
\begin{figure}[hbt!] 
 \begin{tikzpicture}

    \node (Q1) at (-4,0) {$\Q_1$};
    \node (Q2) at (-1,2) {$\Q_1\left( \sqrt{p_1} \right)$};
    \node (Q3) at (-4,2) {$k_1$};
    \node (Q4) at (-7,2) {$\Q_1\left( \sqrt{p_2qr} \right)$};  
    \node (Q5) at (-4,4) {$k_1\left( \sqrt{p_1} \right)$};
    
    \draw (Q1)--(Q2);
    \draw (Q1)--(Q3); 
    \draw (Q1)--(Q4);
    \draw (Q2)--(Q5);
    \draw (Q3)--(Q5);
    \draw (Q4)--(Q5);

    \node (P1) at (4,0) {$\Q_1$};
    \node (P2) at (1,2) {$\Q_1\left( \sqrt{qr} \right)$};
    \node (P3) at (4,2) {$k_1$};
    \node (P4) at (7,2) {$\Q_1\left( \sqrt{p}\right)$};  
    \node (P5) at (4,4) {$k_1\left( \sqrt{p} \right)$};
    
    \draw (P1)--(P2);
    \draw (P1)--(P3); 
    \draw (P1)--(P4);
    \draw (P2)--(P5);
    \draw (P3)--(P5);
    \draw (P4)--(P5); 
   
     \end{tikzpicture}
    \end{figure}
\end{center}

Let $\p_1, \p_2, \q$ and $\mathfrak{r}$ be the prime ideals above $p, q$ and $r$ in $\Q_1$. Suppose $\ell_{11}$ and $\ell_{12}$ denote the primes above $\langle \sqrt{2} \rangle$ in $k_1$. The primes $\p_2, \q$ and $\mathfrak{r}$ are unramified in $\Q_1(\sqrt{p_1})/\Q_1$, but ramified in $k_1/\Q_1$ and $\Q_1(\sqrt{p_2qr})/\Q_1$. Thus, the primes above $p_2, q$, and $r$ in $k_1$ are unramified in $k_1(\sqrt{p_1})$. The ideal $\p_1$ is ramified in $\Q_1(\sqrt{p_1})/\Q_1$ and $k_1/\Q_1$, but unramified in $\Q_1(\sqrt{p_2qr})/\Q_1$. Therefore, the prime above $\p_1$ in $k_1$ is unramified in $k_1(\sqrt{p_1})$. The ideal $\langle \sqrt{2} \rangle$ splits in $k_1/\Q_1$, and remains inert in $\Q_1(\sqrt{p_1})$ and $\Q_1(\sqrt{p_2qr})$. Therefore, the primes $\ell_{11}$ and $\ell_{12}$ are inert in $k_1(\sqrt{p_1})/k_1$. Similarly, $k_1(\sqrt{p})/k_1$ is an unramified extension. But, $p,qr\equiv 1 \Mod{8}$ imply that  $\ell_{11}$ and $\ell_{12}$ are totally split in $k_1(\sqrt{p})/k_1$. Thus, our claim stands true. 

Moreover, $k_1(\sqrt{p_1}, \sqrt{p})/k_1$ is an unramified extension with Galois group isomorphic to $\Z/2\Z \oplus \Z/2\Z$. From Corollary \ref{A(k1)=(2,2)}, it is clear that $k_1(\sqrt{p_1}, \sqrt{p})$ must be the 2-Hilbert class field of $k_1$. In this extension, the primes $\ell_{11}$ and $\ell_{12}$ are not totally split, and the largest subfield where they are split is $k_1(\sqrt{p})$. Hence, $A^{\prime}(k_1)$ is isomorphic to $\Z/2\Z$. From Remark  \ref{A'(k0) =2}, $\#A^{\prime}(k_0) = \#A^{\prime}(k_1) = 2$. The extension $k_n/k$ is totally ramified at all primes above $2$ for each $n \geq 1$. Appealing to part (1) of Theorem \ref{A' stability}, $\#A^{\prime}(k_n) = 2$ for all $n\geq 0$. Therefore, $X^{\prime}(k_{\infty})$ is isomorphic to $\Z/2\Z$. 
$\hfill\Box$

\subsection*{ Proof of Corollary \ref{lambda = 0}} From Theorem \ref{X' is bdd}, we note that for each $n \geq 0$, $A(k_n)/D(k_n) \cong \Z/2\Z$. It is evident that the order of $A(k_n)$ is bounded if the order of $D(k_n)$ is bounded as $n$ tends to infinity. For each $n$, $k_n/k$ is a cyclic extension of degree $2^n$. Let $\tau_n$ be the generator of the Galois group of $k_n/k$ induced by the topological generator $\gamma$ of $k_{\infty}/k$. As the prime above $2$ splits in $k_n/\Q_n$ for all $n$, there exist prime ideals $\ell_{n1}$ and $\ell_{n2}$ in $k_n$ lying above 2, and above the ideals $\ell_{01}$ and $\ell_{02}$ of $k$ in particular. As $k_n/k$ is totally ramified at $\ell_{01}$ and $\ell_{02}$, for any $\tau$ in Gal($k_n/k$), $\ell_{n1}^{\tau} = \ell_{n1}$ and $\ell_{n2}^{\tau} = \ell_{n2}$. This holds for $\tau_n$ as well, and thus, $[\ell_{n1}]^{\tau_n} = [\ell_{n1}]$ and $[\ell_{n2}]^{\tau_n} = [\ell_{n2}]$. Therefore, $D(k_n) = \langle [\ell_{n1}], [\ell_{n2}] \rangle$ must be contained in $B(k_n)$ (as defined before Proposition \ref{Bn bdd}). From Proposition \ref{Bn bdd}, the order of $B(k_n)$ is bounded, and hence the order of $D(k_n)$ must be bounded as $n$ tends to infinity. Consequently, the order of $A(k_n)$ must stabilise for sufficiently large $n$ and the Iwasawa invariant $\lambda$ must be equal to $0$.

\subsection*{Proof of Corollary \ref{4-rank of A(kn)}} For each $n$, $A(k_n), D(k_n)$ and $A^{\prime}(k_n)$ are modules over $\Z/2\Z$, and we have the exact sequence
$$1 \longrightarrow D(k_n) \longrightarrow A(k_n) \longrightarrow A^{\prime}(k_n) \longrightarrow 1.$$

As $\Z/2\Z$ is an integral domain, $ {\rm{rank}_2}A(k_n) = {\rm{rank}_2}D(k_n) + {\rm{rank}_2}A^{\prime}(k_n)$. Since $X^{\prime}(k_{\infty}) \cong \Z/2\Z$, $A^{\prime}(k_n)$ has rank $1$ for all $n \geq 0$. For all $n \geq 1$,  $ {\rm{rank}_2}A(k_n) = 2$ by Lemma \ref{rank A(Kn)=2}. Combining these yields ${\rm{rank}_2}D(k_n) = 1$ for all $n \geq 1$. For $n = 1$ in particular, we have already shown that $A(k_1)$ is isomorphic to $\Z/2\Z \oplus \Z/2\Z$ and $A^{\prime}(k_1)$ is isomorphic to $\Z/2\Z$. Hence, $D(k_1)$ is cyclic, and of order $2$.

For each $n \geq 1$, there exist integers $a_n, b_n$ and $c_n$, all at least equal to $1$ such that $A(k_n)$ and $D(k_n)$ are isomorphic to $\Z/2^{a_n}\Z \oplus \Z/2^{b_n}\Z$ and $\Z/2^{c_n}\Z$ respectively. Without loss of generality, suppose $c_n \leq a_n$. Then, by Theorem \ref{X' is bdd}, $$\Z/2\Z \cong A^{\prime}(k_n) = A(k_n)/D(k_n) \cong \Z/2^{a_n-c_n}\Z\oplus \Z/2^{b_n}\Z.$$

As there is a drop in rank of the quotient by $1$, $a_n = c_n$, and $b_n = 1$ for all $n \geq 1$. It is thus implied that $A(k_n) \cong \Z/2\Z \oplus \Z/2^{a_n}\Z$ for all $n \geq 1$. The subgroup $D(k_n)$ is the largest cyclic subgroup of $A(k_n)$. As the Iwasawa invariant $\lambda$ is equal to $0$, there must exist $n_0 \geq 1$ such that $a_n = a_{n_0}$ for all $n \geq n_0$.

The ideals $\ell_{01}$ and $\ell_{02}$ in $k$ that divide $2$ are principal (see Remark 4.2). Hence, for each $n \geq 1$, $\ell_{n1}^{2^n}$ and $\ell_{n2}^{2^n}$ are principal. Also, since the prime ideal above $2$ in $\Q_n$ is principal, as classes, $[\ell_{n1}]^{-1} = [\ell_{n2}]$ in $A(k_n)$. Therefore, $D(k_n) = \langle [\ell_{n1}] \rangle$ must be of order at most $2^n$, and we conclude that $a_n \leq n$. The $4$-rank of $A(k_n)$ is equal to the $2$-rank of $2A(k_n)/4A(k_n)$. It is apparent from the structure of $A(k_n)$ that its $4$-rank is at most equal to $1$ for all $n \geq 0$. Combining all, we can conclude that $X(k_{\infty})$ is isomorphic to $\Z/2\Z \oplus \Z/2^{a_{n_0}}\Z$. This completes the proof of the corollary. $\hfill\Box$

\smallskip

{\bf Acknowledgements.} The authors take immense pleasure to thank Indian Institute of Technology Guwahati for providing excellent facilities to carry out this research. The research of the second author is partially funded by the MATRICS, SERB research grant MTR/2020/000467.


\begin{thebibliography}{9999}
\bibitem{Brumer}
A. Brumer, {\it On the units of algebraic number fields}, {\sf Mathematika}, {\bf 14} (1967), 121-124.

\bibitem{chevalley}
C. Chevalley, {\it Sur la th\'{e}orie du corps de classes dans les corps finis et les corps locaux (Th\'{e}se)}, {\sf J. Faculty of Sciences Tokyo}, {\bf 2} (1933), 365-476.

\bibitem{nancy_childress}
N. Childress, {\it Class Field Theory}, {\sf Springer Science and Business Media}, (2008).

\bibitem{ferrero-washington}
B. Ferrero and L C. Washington, {\it The Iwasawa invariant $\mu_{p}$ vanishes for abelian number fields}, {\sf Annals of Mathematics}, {\bf 109} (1979), 377-395.

\bibitem{fukuda}
T. Fukuda, { \it Remarks on $\Z_p $-extensions of number fields}, {\sf Proceedings of the Japan Academy, Series A, Mathematical Sciences}, {\bf 70} (1994), 264-266.

\bibitem{fukuda-komatsu}
T. Fukuda and K. Komatsu, {\it On the Iwasawa $\lambda$-Invariant of the Cyclotomic $\Z_2$-Extension of a Real Quadratic Field}, {\sf Tokyo Journal of Mathematics}, {\bf 28} (2005), 259-264.

\bibitem{RoyCFT}
R. Fuller, {\it Roy Fuller's notes from Langlands' course on class field theory, Spring 1964}, {\sf https://personal.math.ubc.ca/$\sim$cass/fuller/fuller.html}. 

\bibitem{greenberg}
R. Greenberg, {\it  On the Iwasawa invariants of totally real number fields}, {\sf Amer. J. Math.}, {\bf 98} (1976), 263-284.

\bibitem{iwasawa1}
K. Iwasawa, {\it On $\Gamma$-extensions of algebraic number fields}, {\sf Bull of Amer. Math. Soc.}, {\bf 65} (1959), 183-226.

\bibitem{iwasawa}
K. Iwasawa, {\it On $\Z_{\ell}$-extensions of algebraic number fields}, {\sf Annals of Math.}, {\bf 98} (1973), 246-326.

\bibitem{janusz}
G. J. Janusz, {\it Algebraic Number Fields}, {\sf American mathematical Society}, {\bf 7} (1996).

\bibitem{kumakawa}
N. Kumakawa, {\it On the Iwasawa $\lambda$-invariant of the cyclotomic $\Z_2$-extension of $\Q(\sqrt{pq})$ and the $2$-part of the class number of $\Q(\sqrt{pq}, \sqrt{2 + \sqrt{2}})$}, {\sf International Journal of Number Theory}, {\bf 17} (2021), 931-958.

\bibitem{LS}
H Laxmi and A. Saikia, {\it  $\Z_2$-extension of real quadratic fields with $\Z/2\Z$ as $2$-class group at each layer}, {\sf To appear in The Ramanujan Journal}, DOI: 10.1007/s11139-024-00869-8, arXiv:2310.03543. 

\bibitem{mouhib}
A. Mouhib, {\it The structure of the unramified abelian Iwasawa module of some number fields}, {\sf Pacific Journal of Mathematics}, {\bf 323} (2023), 173-184.

\bibitem{mouhib-mova}
A. Mouhib and A. Movaheddi, {\it Cyclicity of the unramified Iwasawa module}, {\sf Manuscripta Mathematica}, {\bf 135} (2011), 91-106.

\bibitem{mizusawa1}
Y. Mizusawa, {\it On the Iwasawa invariants of $\Z_2$-extensions of certain real quadratic fields}, {\sf Tokyo Journal of Mathematics}, {\bf 27} (2004), 255-261.

\bibitem{Mizusawa ProcAMS}
Y. Mizusawa, {\it On unramified galois 2-groups over $\Z_2$-extensions of real quadratic fields}, {\sf Proceedings of the American Mathematical Society}, {\bf 138} (2010), 3095-3103.

\bibitem{Mizusawa CMB}
Y. Mizusawa, {\it On metabelian 2-class field towers over $\Z_2$-extensions of real quadratic fields}, {\sf Canadian Math. Bull,}, {\bf 65} (2022), 795-805.

\bibitem{nishino}
Y. Nishino, {\it On the Iwasawa Invariants of the Cyclotomic $\Z_2$-Extensions of Certain Real Quadratic Fields}, {\sf Tokyo Journal of Mathematics}, {\bf 29} (2006), 239-245.

\bibitem{ozaki-taya}
M. Ozaki and H. Taya, {\it On the Iwasawa $\lambda_2$-invariants of certain families of real quadratic fields}, {\sf Manuscripta Math.}, {\bf 94} (1997), 437-444.

\bibitem{sharifi}
R. Sharifi, {\it Iwasawa theory}, {\sf https://www.math.ucla.edu/$\sim$sharifi/iwasawa.pdf}. 

\bibitem{washington_book}
L C. Washington, {\it Introduction to cyclotomic fields},  {\sf Springer Science \& Business Media}, {\bf 83} (1997).
\end{thebibliography}
\end{document}